\documentclass[letterpaper,onecolumn]{article} 
\usepackage{amsmath}
\usepackage{amsfonts}
\usepackage{verbatim}
\usepackage{subfigure}
\usepackage{graphicx}
\usepackage[hmargin=2.5cm,vmargin=2.5cm]{geometry}
\usepackage[siunitx,american]{circuitikz}
\newcommand{\sech}{\mathrm{sech} \,}

\author{S. Yusef Shafi \thanks{Department of Electrical Engineering and Computer Sciences, University of California, Berkeley. Email: yusef@eecs.berkeley.edu.}} 
\title{Guaranteeing Spatial Uniformity in Diffusively-Coupled Systems} 

\newtheorem{theorem}{Theorem}[section]
\newtheorem{lemma}[theorem]{Lemma}
\newtheorem{proposition}[theorem]{Proposition}
\newtheorem{corollary}[theorem]{Corollary}

\newenvironment{proof}[1][Proof]{\begin{trivlist}
\item[\hskip \labelsep {\bfseries #1}]}{\end{trivlist}}

\begin{document} 
\maketitle 

\begin{abstract}
We present a condition that guarantees spatially uniformity in the solution trajectories of a diffusively-coupled compartmental ODE model, where each compartment represents a spatial domain of components interconnected through diffusion terms with like components in different compartments. Each set of like components has its own weighted undirected graph describing the topology of the interconnection between compartments. The condition makes use of the Jacobian matrix to describe the dynamics of each compartment as well as the Laplacian eigenvalues of each of the graphs. We discuss linear matrix inequalities that can be used to verify the condition guaranteeing spatial uniformity, and apply the result to a coupled oscillator network. Next we turn to reaction-diffusion PDEs with Neumann boundary conditions, and derive an analogous condition guaranteeing spatial uniformity of solutions. The paper contributes a relaxed condition to check spatial uniformity that allows individual components to have their own specific diffusion terms and interconnection structures.
\end{abstract}

\section{Introduction}
Diffusively coupled models are crucial to understanding the dynamical behavior of a range of engineering and biological systems. Spatially distributed systems whose individual compartments interact with each another over weighted undirected graphs are naturally modeled as diffusively-coupled systems of ordinary differential equations (ODEs). Coupled ring oscillators constitute a class of voltage-controlled oscillators frequently used in applications such as clock recovery circuits and disk-drive read channels \cite{GeEtAl} \cite{DeVitoEtAl} \cite{Negahban}. Understanding synchronization behavior in ring oscillator circuits requires the tools of nonlinear analysis. Turning to the biological context, one of the major ideas behind pattern formation in cells and organisms is based on diffusion-driven instability \cite{SegelJackson} \cite{Turing}. The behavior is observed when higher-order spatial modes in a reaction-diffusion partial differential equation (PDE) are destabilized by diffusion, resulting in the growth of spatial inhomogeneities, and has been studied extensively \cite{Murray} \cite{CrossHohenberg} \cite{OthmerEtAl} \cite{HsiaEtAl}.

We begin our discussion in Section 2 by studying compartmental ODE models, where each compartment represents a well-mixed spatial domain wherein like components in different compartments are coupled by diffusion \cite{Hale}. In contrast to previous work \cite{Arcak} where the interconnections between compartments were identical for all components, we allow diffusion terms to be unique to each component. In particular, the interconnections within each set of like components may be described by its own weighted undirected graph. We derive Lyapunov inequality conditions that guarantee spatial homogeneity in the trajectories of each component. We then derive convex linear matrix inequality \cite{BoydEtAl} tests as in \cite{Arcak} that can be used to certify the Lyapunov inequality conditions. In Section 3, we apply the LMI tests to study the behavior of a coupled ring oscillator circuit. We emphasize that each node in a component may have its own set of neighboring nodes to which it is diffusively coupled independent of the set of neighbors of other nodes in the same compartment. 

We next turn to reaction-diffusion PDEs with Neumann boundary condition in Section 4, and establish a spatial homogeneity condition analogous to the result for coupled oscillators. The condition we derive allows for elliptic operators that model diffusion that may vary spatially and between species \cite{Smoller}, whereas previous work has studied the special case of the Laplacian operator \cite{AshkenaziOthmer}\cite{ConwayEtAl}\cite{JonesSleeman}\cite{Othmer}\cite{Arcak}. We conclude in Section 5.

\section{Spatial uniformity in diffusively coupled systems of ODEs}
We begin by considering a compartmental ODE model where each compartment represents a spatial domain interconnected with the other compartments over an undirected graph:
\begin{equation}
\dot{x}_{i,k} = f(x_i)_k + \sum_{j \in N_{i,k}} w_{ij}^{(k)}(x_{j,k} -x_{i,k}),\,\,\, i = 1,\ldots,N.
\label{eq:OscComp}
\end{equation}
The vector $x_i \in \mathbb{R}^n$ is the state of the $i$-th compartment, the vector field $f(x_i)_k$ is the $k$-th component of the vector field $f(x_i)$ acting on $x_i$, the set $\mathcal{N}_{i,k}$ consists of the neighbors of the $k$-th component of compartment $i$, and the scalar $w_{ij}^{(k)}=w_{ji}^{(k)} \in \mathbb{R}$ is a weighting factor. We aggregate the dynamics of each compartment using the stacked vector $X = [ x_1^T \ldots \,x_N^T ]^T$ and represent the interconnections between like components in different compartments by a generalized symmetric positive semidefinite graph Laplacian matrix $L_k \in \mathbb{R}^{N \times N}$:
\begin{equation}
\dot{X} = F(X) - \left ( \sum_{k=1}^n L_k \otimes E_k \right ) X,
\label{eq:OscSys}
\end{equation}
where $E_k = e_k e_k^T \in \mathbb{R}^{n \times n}$ is the product of the $k$-th standard basis vector $e_k$ multiplied by its transpose, and $F(X) = [f(x_1)^T \ldots f(x_N)^T]^T$. In the event that the $k$-th set of like components are not interconnected with one another, we set $L_k=0$. Define $\lambda_2^{(k)}$ as the second smallest eigenvalue of $L_k$, and note that since $L_k 1_N=0$,
\begin{equation}
z^T L_k z \geq \lambda_2^{(k)} z^T z
\label{eq:algebraicConnectivity}
\end{equation}
for all $z \in \mathbb{R}^n$ with $z \perp 1_N$. Let $J(x)=\frac{\partial f}{\partial x}\left |_x \right.$ denote the Jacobian of $f(x)$ at $x$.

\begin{proposition}
Consider the system (\ref{eq:OscSys}). Suppose there exists a convex set $\mathcal{X} \in \mathbb{R}^n$, a positive definite matrix $P$, and a constant $\epsilon>0$ such that the following conditions hold:
\begin{align}
&P \left (J(x)-\sum_{k=1}^n \lambda_2^{(k)} E_k \right ) + \left ( J(x)-\sum_{k=1}^n \lambda_2^{(k)} E_k \right)^T P \preceq - \epsilon I\,\,\, \forall x \in \mathcal{X} \label{eq:coroconditions1}
 \\
&PE_k+E_kP \succeq 0 \text{  for each } k \in \{1,\ldots,n\} \text{ with } L_k \ne 0.
\label{eq:coroconditions2}
\end{align}
Then for any pair $(i,j) \in \{1,\ldots,N\} \times \{1,\ldots,N\}$ and any index $k \in \{1,\ldots,n\}$, we have:
$x_{i,k}(t) - x_{j,k}(t) \rightarrow 0$
exponentially as $t \rightarrow \infty$. \hfill $\square$
\label{prop1}
\end{proposition}

\begin{proof}
First recall that $z^T L_k z \geq \lambda_2^{(k)} z^T z$ for all $z \perp 1_N$, and that $z^T (L_k \otimes I_n) z \geq \lambda_2^{(k)} z^T z$ for all $z \perp 1_N \otimes I_n$. Define the following terms: 
\begin{equation}
\begin{aligned}
\bar{x} = \frac{1}{N} \sum_{i=1}^N x_i &= \frac{1}{N} (1_N^T \otimes I_n) X \\
\bar{X} &= 1_N \otimes \bar{x} \\
\tilde{x}_i &= x_i - \bar{x} \\
\tilde{X} &= X - \bar{X}.
\end{aligned}
\end{equation}

Since $\sum_{i=1}^N \tilde{x}_i = 0$, it holds that $\tilde{X}^T (1_N \otimes M) = 0$ for all matrices $M$ with n rows. The dynamics of $\tilde{X}$ are given by:
\begin{equation}
\begin{aligned}
\dot{\tilde{X}} &= F(X) - \dot{\bar{X}} - L X \\
&= F(X) - \dot{\bar{X}} - L \tilde{X},
\end{aligned}
\end{equation}
where $L = \sum_{k=1}^n L_k \otimes E_k$. We differentiate the candidate Lyapunov function $V= \frac{1}{2} \tilde{X}^T(I_N \otimes P) \tilde{X}$:
\begin{equation}
\begin{aligned}
\dot{V} &= \tilde{X}^T(I_N \otimes P) (F(X)-\dot{\bar{X}}) - \tilde{X}^T (I_N \otimes P) L \tilde{X} \\
&= \tilde{X}^T(I_N \otimes P) (F(X)-\dot{\bar{X}}) - \tilde{X}^T \sum_{k=1}^n (L_k \otimes P E_k) \tilde{X}.
\end{aligned}
\end{equation}
We observe that
\begin{equation}
(L_k \otimes PE_k) + (L_k \otimes PE_k)^T = L_k \otimes (PE_k + E_k P),
\label{eq:laplacianSymmetric}
\end{equation}
and that because condition (\ref{eq:coroconditions2}) holds, there exists a matrix $Q_k$ such that $Q_k^TQ_k = \frac{1}{2}(P E_k + E_k P)$. Then
\begin{equation}
\begin{aligned}
\tilde{X}^T (L_k \otimes P E_k) \tilde{X} &= \tilde{X}^T (I_N \otimes Q_k^T) (L_k \otimes I_n) (I_N \otimes Q_k) \tilde{X} \\
&= y_k^T (L_k \otimes I_n) y_k,
\end{aligned}
\end{equation}
where $y_k = (I_N \otimes Q_k)\tilde{X}$. Because of the orthogonality relation $y_k \perp 1_N \otimes I_n$ and condition (\ref{eq:algebraicConnectivity}),  it follows that
\begin{equation}
\begin{aligned}
\tilde{X}^T (I_N \otimes P) L \tilde{X} &= \sum_{k=1}^n y_k^T (L_k \otimes I_n) y_k \\
& \geq \sum_{k=1}^n \lambda_2^{(k)} y_k^T y_k \\
&= \sum_{k=1}^n \lambda_2^{(k)} \tilde{X}^T (I_N \otimes P E_k) \tilde{X} \\
&= \sum_{k=1}^n \lambda_2^{(k)} \sum_{i=1}^N \tilde{x}_i^T P E_k \tilde{x}_i.
\end{aligned}
\end{equation}
Defining $F(\bar{X})=1_N \otimes f(\bar{x})$ and adding and subtracting $\tilde{X}(I_N \otimes P)F(\bar{X})$, we have:
\begin{equation}
\begin{aligned}
\dot{V} &\leq \tilde{X}^T(I_N \otimes P) (F(X)-\dot{\bar{X}}) - \sum_{k=1}^n \lambda_2^{(k)} \sum_{i=1}^N \tilde{x}_i^T P E_k \tilde{x}_i \\
&= \tilde{X}^T(I_N \otimes P) (F(X)-F(\bar{X})) + \tilde{X}^T(I_N \otimes P)(1_N \otimes (f(\bar{x}) - \dot{\bar{x}})) \\
&\,\,\,\,\,\,- \sum_{k=1}^n \lambda_2^{(k)} \sum_{i=1}^N \tilde{x}_i^T P E_k \tilde{x}_i \\
&= \tilde{X}^T(I_N \otimes P) (F(X)-F(\bar{X})) + \tilde{X}^T(1_N \otimes P(f(\bar{x}) - \dot{\bar{x}})) \\
&\,\,\,\,\,\,- \sum_{k=1}^n \lambda_2^{(k)} \sum_{i=1}^N \tilde{x}_i^T P E_k \tilde{x}_i.
\end{aligned}
\end{equation}
Recalling that $\bar{X}^T (1_N \otimes M) = 0$, we take $M=P(f(\bar{x}) - \dot{\bar{x}})$ and apply the mean value theorem:
\begin{equation}
\begin{aligned}
\dot{V} & \leq \left ( \sum_{i=1}^N \tilde{x}_i^T P \left (f(x_i)-f(\bar{x}) \right ) \right ) - \left ( \sum_{k=1}^n \lambda_2^{(k)} \sum_{i=1}^N \tilde{x}_i^T P E_k \tilde{x}_i \right ) \\
&= \sum_{i=1}^N \tilde{x}_i^T \left ( P (f(x_i)-f(\bar{x})) - P \sum_{k=1}^n \lambda_2^{(k)} E_k \right ) \tilde{x}_i \\
&= \sum_{i=1}^N \int_0^1 \tilde{x}_i^T P \left (J(\bar{x}+s\tilde{x}_i) - \sum_{k=1}^n \lambda_2^{(k)}E_k \right ) \tilde{x}_i \, ds.
\end{aligned}
\end{equation}
Because condition (\ref{eq:coroconditions1}) holds, we have:
\begin{equation}
\dot{V} \leq -\frac{\epsilon}{2} \tilde{X}^T\tilde{X} \leq -\frac{\epsilon}{\lambda_{\text{max}}(P)} V,
\end{equation}
which concludes the proof.
\hfill $\square$ \\
\end{proof}

The additional condition that the product $P E_k$ be symmetric for each $k \in \{1,\ldots,n\}$ with $L_k \ne 0$ allows us to generalize Proposition \ref{prop1}, as in \cite{Arcak}, to handle non-symmetric generalized graph Laplacians (e.g., where $w_{ij}^{(k)} \ne w_{ji}^{(k)}$). In place of (\ref{eq:laplacianSymmetric}), take
\begin{equation}
(L_k \otimes PE_k) + (L_k \otimes PE_k)^T = (L_k+L_k^T) \otimes (PE_k),
\end{equation}
which holds when $PE_k$ is symmetric, and define $\lambda_2^{(k)}$ as the smallest positive number such that (\ref{eq:algebraicConnectivity}) holds.

In order to check the conditions of Proposition \ref{prop1}, we note two corollaries that follow from (\cite{Arcak}, Theorems 2 and 3), where the Jacobian matrix over the convex set $\mathcal{X}$ is itself parametrized by a convex set.

\begin{corollary}
If there exist constant matrices $Z_1,\ldots,Z_q$ and $S_l,\ldots,S_m$ such that
\begin{equation}
J(x) \in \text{conv}\{Z_1,\ldots,Z_q\} + \text{cone}\{S_l,\ldots,S_m\} \,\,\, \forall \, x \in \mathcal{X},
\end{equation}
then the existence of a symmetric matrix $P$ satisfying
\begin{equation}
\begin{aligned}
P \left (Z_k-\sum_{k=1}^n \lambda_2^{(k)}E_k \right ) + \left (Z_k-\sum_{k=1}^n \lambda_2^{(k)}E_k \right )^TP &\prec 0, \text{   } k = 1,\ldots,q \\
PS_k+S_k^TP &\preceq 0, \text{   } k=1,\ldots,m
\end{aligned}
\label{eq:lmiconditions}
\end{equation}
implies condition (\ref{eq:coroconditions1}) for some $\epsilon > 0$. If $J(x)$ is surjective onto $\text{conv}\{Z_1,\ldots,Z_q\} + \text{cone}\{S_l,\ldots,S_m\}$, then the converse is true. \hfill $\square$
\label{coro1}
\end{corollary}

Next, we define a \textit{convex box} as:
\begin{equation}
\text{box} \{M_0,M_1,\ldots,M_p\} = \{M_0+\omega_1 M_1 + \ldots  + \omega_p M_p \, | \, \omega_k \in [0,1] \text{ for each } k=1,\ldots,p \}.
\end{equation}

\begin{corollary}
Suppose that $J(x)$ is contained in a convex box:
\begin{equation}
J(x) \in \text{box} \{A_0,A_1,\ldots,A_l\} \,\,\, \forall \, x \in \mathcal{X},
\label{eq:Box}
\end{equation}
where $A_1,\ldots,A_l$ are rank-one matrices that can be written as $A_i = B_iC_i^T$, with $B_i,\, C_i \in \mathbb{R}^n$. If there exists a positive definite matrix $\mathcal{P}$ with:
\begin{equation}
\mathcal{P} = \left [ \begin{array}{cccc} P & 0 & \ldots & 0 \\ 0 & q_l & 0 & 0 \\ \vdots & \ddots & \ddots & \vdots \\ 0 & \hdots & 0 & q_l \end{array} \right ], \text{   } P \in \mathbb{R}^{n \times n}, \, q_i \in \mathbb{R}, \, i = 1,\ldots,l,
\end{equation}
satisfying:
\begin{equation}
\mathcal{P} \left [ \begin{array}{cc} A_0 - \sum_{k=1}^n \lambda_2^{(k)}E_k & B \\ C^T & -I_n \end{array} \right ] + \left [ \begin{array}{cc} A_0 - \sum_{k=1}^n \lambda_2^{(k)}E_k & B \\ C^T & -I_n \end{array} \right ]^T \mathcal{P} \prec 0,
\label{eq:numericalconditions}
\end{equation}
with $B = [B_1 \ldots B_l]$ and $C=[C_1 \ldots C_l]$, then the upper left (positive definite) principal submatrix $P$ satisfies condition (\ref{eq:coroconditions1}) for some $\epsilon > 0$. If $l = 1$ and the image of $\mathcal{X}$ is surjective onto $\text{box}\{A_0,A_1\}$, then the converse is true. \hfill $\square$
\label{coro2}
\end{corollary}

\section{Ring Oscillator Circuit Example}

\begin{figure}[h]
\centering
\scalebox{0.65}{\begin{circuitikz}[scale=1.5]\draw

(1,3) node[not port] (not1) {}
(3,3) node[not port] (not2) {}
(5,3) node[not port] (not3) {}
(6.7,3) node (node1) {}
(not1.out) to [R=$R_1$] (not2.in)
++(0,-1) node[ground] {} to [C=$C_1$] (not2.in)
(not2.out) to [R=$R_2$] (not3.in)
++(0,-1) node[ground] {} to [C=$C_2$] (not3.in)
(not3.out) to [R=$R_3$] (node1)
++(0,-1) node[ground] {} to [C=$C_3$] (node1)
(node1) to [short] (7,3) to [short] (7,4) to [short] (.5,4) to [short] (.5,3) to [short] (not1.in)

(1,0) node[not port] (not4) {}
(3,0) node[not port] (not5) {}
(5,0) node[not port] (not6) {}
(6.7,0) node (node2) {}
(not4.out) to [R=$R_1$] (not5.in)
++(0,-1) node[ground] {} to [C=$C_1$] (not5.in)
(not5.out) to [R=$R_2$] (not6.in)
++(0,-1) node[ground] {} to [C=$C_2$] (not6.in)
(not6.out) to [R=$R_3$] (node2)
++(0,-1) node[ground] {} to [C=$C_3$] (node2)
(node2) to [short] (7,0) -- (7,0) to [short] (7,1) to [short] (.5,1) to [short] (.5,0) to [short] (not4.in)

(1,-3) node[not port] (not7) {}
(3,-3) node[not port] (not8) {}
(5,-3) node[not port] (not9) {}
(6.7,-3) node (node3) {}
(not7.out) to [R=$R_1$] (not8.in)
++(0,-1) node[ground] {} to [C=$C_1$] (not8.in)
(not8.out) to [R=$R_2$] (not9.in)
++(0,-1) node[ground] {} to [C=$C_2$] (not9.in)
(not9.out) to [R=$R_3$] (node3)
++(0,-1) node[ground] {} to [C=$C_3$] (node3)
(node3) to [short] (7,-3) -- (7,-2) -- (.5,-2) -- (.5,-3) -- (not7.in)

(not2.in) -- (2.53,3.5) -- (.2,3.5) to [R=$R^{(1)}$] (.2,.5) -- (2.53,.5) -- (not5.in)
(.2,.5) to [R=$R^{(1)}$] (.2,-2.5) -- (2.53,-2.5) -- (not8.in)
(.2,-2.5) -- (-.2,-2.5) to [R=$R^{(1)}$] (-.2,3.5) -- (.2,3.5)

{[anchor=south] (2.53,3.5) node {$x_{1,1}$} (2.53,.5) node {$x_{2,1}$} (2.53,-2.5) node {$x_{3,1}$}}

(not3.in) -- (4.53,3.5) -- (3.53,3.5) to [R=$R^{(2)}$] (3.53,.5) -- (4.53,.5) -- (not6.in)
(3.53,.5) to [R=$R^{(2)}$] (3.53,-2.5) -- (4.53,-2.5) -- (not9.in)

{[anchor=south] (4.53,3.5) node {$x_{1,2}$} (4.53,.5) node {$x_{2,2}$} (4.53,-2.5) node {$x_{3,2}$}}
;\end{circuitikz}}
\caption{\small{Example of a three-stage ring oscillator circuit as in (\ref{eq:ringOsc}) coupled through nodes $1$ and $2$.}}
\label{fig:ringOscCir}
\end{figure}
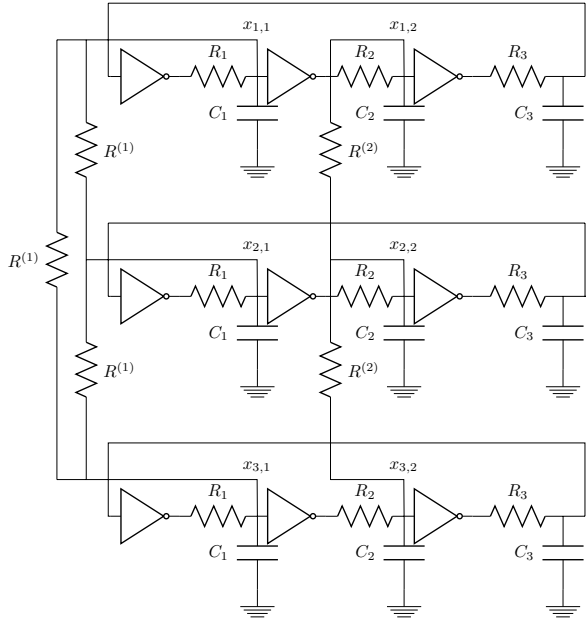

Consider the $n$-stage ring oscillator whose dynamics are given by:
\begin{equation}
\begin{aligned}
\dot{x}_{i,1} &= - \eta_1 x_{i,1} - \alpha_1 \tanh (\beta_1 x_{i,n}) + w_{i,1}\\
\dot{x}_{i,2} &= - \eta_2 x_{i,2} + \alpha_2 \tanh (\beta_2 x_{i,1}) + w_{i,2}\\
& \,\,\, \vdots \\
\dot{x}_{i,n} &= - \eta_n x_{i,n} + \alpha_n \tanh (\beta_n x_{i,n-1}) + w_{i,n},
\end{aligned}
\label{eq:ringOsc}
\end{equation}
with coupling between corresponding nodes of each circuit. The parameters $\eta_k= \frac{1}{R_k C_k}$, $\alpha_k$, and $\beta_k$ correspond to the gain of each inverter. The input is given by:
\begin{equation}
w_{i,k} = d_k \sum_{j \in \mathcal{N}_{i,k}} (x_{j,k} -  x_{i,k}),
\label{eq:ringOscInput}
\end{equation}
where $d_k = \frac{1}{R^{(k)} C_k}$ and $\mathcal{N}_{i,k}$ denotes the nodes to which node $k$ of circuit $i$ is connected. We wish to determine if the solution trajectories of each set of like nodes of the coupled ring oscillator circuit given by (\ref{eq:ringOsc})-(\ref{eq:ringOscInput}) synchronize, that is:
\begin{equation}
x_{i,k}-x_{j,k} \rightarrow 0 \text{ exponentially as } t \rightarrow \infty
\end{equation}
for any pair $(i,j) \in \{1,\ldots,N\} \times \{1,\ldots,N\}$ and any index $k \in \{1,\ldots,n\}$.

For clarity in our discussion, we take $n=3$ as in Figure \ref{fig:ringOscCir}, noting that the derivation is identical for any choice of $n$. We first write the Jacobian of the system (\ref{eq:ringOsc}), where we have omitted the subscripts indicating circuit membership:
\begin{equation}
J(x) = \left [ \begin{array}{ccc} -\eta_1 & 0 & -\alpha_1 \beta_1 \sech^2 (\beta_1 x_3) \\
\alpha_2 \beta_2 \sech^2(\beta_2 x_1) & -\eta_2 & 0 \\
0 & \alpha_3 \beta_3 \sech^2(\beta_3 x_{2}) & -\eta_3 \end{array} \right ].
\end{equation}
Define the matrices
\begin{equation}
A_0 = \left [ \begin{array}{ccc} -\eta_1 & 0 & 0 \\
0 & -\eta_2 & 0 \\
0 & 0 & -\eta_3 \end{array} \right ] \text{   } 
A_1 = \left [ \begin{array}{ccc} 0 & 0 & -\alpha_1 \beta_1 \\
0 & 0 & 0 \\
0 & 0 & 0 \end{array} \right ] \text{   }
A_2 = \left [ \begin{array}{ccc} 0 & 0 & 0 \\
\alpha_2 \beta_2 & 0 & 0 \\
0 & 0 & 0 \end{array} \right ] \text{   }
A_3 = \left [ \begin{array}{ccc} 0 & 0 & 0 \\
0 & 0 & 0 \\
0 & \alpha_3 \beta_3 & 0 \end{array} \right ].
\end{equation}
Then it follows that $J(x)$ is contained in a convex box:
\begin{equation}
J(x) \in \text{box} \{A_0,A_1,A_2,A_3\}.
\label{eq:ringOscBox}
\end{equation}
While the method of Corollary \ref{coro1} involves parametrizing a convex box as a convex hull with $2^p$ vertices, and potentially a prohibitively large linear matrix inequality computation, the problem structure can be exploited using Corollary \ref{coro2} to obtain a simple analytical condition for synchronization of trajectories. In particular, the Jacobian of the ring oscillator exhibits a \textit{cyclic} structure. 
The matrix for which we seek a $\mathcal{P}$ satisfying (\ref{eq:numericalconditions}) is given by:
\begin{equation}
M = \left [ \begin{array}{cc} A_0 - \sum_{k=1}^n \lambda_2^{(k)}E_k & B \\ C^T & -I \end{array} \right ] =
\left [ \begin{array}{rrrrrr} -\eta_1 - \lambda_2^{(1)} & 0 & 0 & 0 & 0 & -\alpha_1 \beta_1 \\
0 & -\eta_2 - \lambda_2^{(2)} & 0 & \alpha_2 \beta_2 & 0 & 0 \\
0 & 0 & -\eta_3 - \lambda_2^{(3)} & 0 & \alpha_3 \beta_3 & 0 \\
1 & 0 & 0 & -1 & 0 & 0 \\
0 & 1 & 0 & 0 & -1 & 0 \\
0 & 0 & 1 & 0 & 0 & -1
\end{array} \right ].
\end{equation}
Note that the matrix $M$ exhibits a cyclic structure, and by a suitable permutation $G$ of its rows and columns, it can be brought into a cyclic form $\tilde{M}=GMG^T$. Since $\tilde{M}$ is cyclic, it is amenable to an application of the \textit{secant criterion} \cite{ArcakSontag}, which implies that the condition
\begin{equation}
\frac{\Pi_{k=1}^3 \alpha_k \beta_k}{\Pi_{l=1}^3 (\eta_l+\lambda_l)} < \sec^3 \left (\frac{\pi}{3} \right )
\label{eq:secantcriterion}
\end{equation}
holds if and only if $\tilde{M}$ satisfies
\begin{equation}
\tilde{\mathcal{P}} \tilde{M} + \tilde{M}^T \tilde{\mathcal{P}} \prec 0
\label{eq:lmidiagonal}
\end{equation}
for some diagonal $\tilde{\mathcal{P}} \succ 0$. Pre- and post-multiplying (\ref{eq:lmidiagonal}) by $G^T$ and $G$, respectively, we have
\begin{equation}
G^T\tilde{\mathcal{P}} G M +M^T G^T \tilde{\mathcal{P}}G \prec 0.
\end{equation}
Note that $G^T\tilde{\mathcal{P}} G$ is diagonal, and so if $\tilde{M}$ is diagonally stable, then $M$ is diagonally stable as well. We conclude that if the secant criterion in (\ref{eq:secantcriterion}) is satisfied, then Corollary \ref{coro2} holds, and so Proposition \ref{prop1} holds, with:
\begin{equation}
x_{i,k}-x_{j,k} \rightarrow 0 \text{ exponentially as } t \rightarrow \infty
\end{equation}
for any pair $(i,j) \in \{1,\ldots,N\} \times \{1,\ldots,N\}$ and any index $k \in \{1,\, 2,\, 3\}$.

We note that the condition for synchrony that we have found recovers Theorem 2 in \cite{GeEtAl}, which makes use of an input-output approach to synchronization \cite{ScardoviEtAl}. We have derived the condition using Lyapunov functions in an entirely different manner from the input-output approach.

\section{Spatial uniformity in reaction-diffusion PDEs}
Consider the connected, bounded domain $\Omega \subseteq \mathbb{R}^r$ with smooth boundary $\partial \Omega$, spatial variable $\xi \in \Omega$, and outward normal vector $n(\xi)$ for $\xi \in \partial \Omega$. We consider elliptic operators $L_k$ given by:
\begin{equation}
L_ku = \nabla \cdot (A_k(\xi) \nabla u), \,\,\, A_k: \Omega \rightarrow \mathbb{R}^{r \times r}, \,\,\, k\in\{1,\ldots,n\}, 
\label{eq:ellipticoperator}
\end{equation}
where the function $A_k$ is symmetric and bounded and $\exists \alpha > 0$ such that for all $\xi \in \Omega$ and for all $\zeta = (\zeta_1,\zeta_2,\ldots,\zeta_r) \in \Omega$, $\sum_{i,j}^r a_{ij}(\xi) \zeta_i \zeta_j \geq \alpha | \zeta |^2$. We will study the reaction-diffusion equation:
\begin{equation}
\frac{\partial x}{\partial t} = f(x) + \mathcal{L} x,
\label{eq:ellipticpde}
\end{equation}
subject to Neumann boundary conditions $\nabla x_i(t,\xi) \cdot n(\xi) = 0 \,\,\, \forall \xi \in \partial \Omega$, where $n(\xi)$ is a vector normal to $\partial \Omega$, $x(t,\xi) \in \mathbb{R}^n$, and 
\begin{equation}
\mathcal{L} x = [\nabla \cdot (A_1(\xi) \nabla x_1) \,\,\, \ldots \,\,\, \nabla \cdot (A_n(\xi) \nabla x_n)]^T
\end{equation}
is a vector of elliptic operators with respect to the spatial variable $\xi$ applied to each entry of $x(t,\xi)$. In a reaction-diffusion system, $x$ represents a vector of concentrations for the reactants. We do not emphasize well-posedness of solutions to reaction-diffusion PDEs: results on existence of solutions to the reaction PDE with $A_k=d_k I$ for each $k$ can be found in \cite{Smoller}.

Define $\pi\{v\} = v - \bar{v}$, where
\begin{equation}
\bar{v} = \frac{1}{|\Omega|}\int_{\Omega} v(\xi) d \xi.
\end{equation}
Recall the $L^2(\Omega)$ inner product
\begin{equation}
\langle u,v \rangle_{L^2(\Omega)} = \int_{\Omega} u^T(\xi) v(\xi) d \xi
\end{equation}
with the norm $||v||_{L^2(\Omega)} = \sqrt{\langle v,v \rangle_{L^2(\Omega)}}$.

We now recall a result following from the Poincar\'e principle as in \cite{Henrot}, which gives a variational characterization of the eigenvalues of an elliptic operator.
\begin{lemma}
Let $\lambda_2^{(k)}$ be the second smallest Neumann eigenvalue of the operator $L_k$ as in (\ref{eq:ellipticoperator}) defined on the connected, bounded domain $\Omega \subseteq \mathbb{R}^r$ with smooth boundary $\partial \Omega$ and spatial variable $\xi \in \Omega$. Let $v = v(\xi)$ be a function not identically zero in $L^2(\Omega)$ with derivatives $\frac{\partial v}{\partial \xi_i} \in L^2(\Omega)$ that satisfies the Neumann boundary condition $\nabla v(\xi) \cdot n(\xi) = 0$ and satisfies $\int_{\Omega} v \, d \xi = 0$. Then the following inequality holds:
\begin{equation}
\int_{\Omega}  \nabla v \cdot (A_k \nabla v) \, d \xi \geq \lambda_2^{(k)} \int_{\Omega} v^2 \, d \xi.
\end{equation}
\hfill $\square$
\label{lemma1}
\end{lemma}
We now show that the solutions of (\ref{eq:ellipticpde}) achieve spatial uniformity under the conditions (\ref{eq:coroconditions1})-(\ref{eq:coroconditions2}):
\begin{proposition}
Consider the system (\ref{eq:ellipticpde}). Suppose there exists a convex set $\mathcal{X} \subseteq \mathbb{R}^n$, positive definite matrix $P$, and constant $\epsilon>0$ such that the conditions (\ref{eq:coroconditions1})-(\ref{eq:coroconditions2}) hold. Then for every classical solution $x(t,\xi) : [0,\infty) \times \Omega \rightarrow \mathcal{X}$, $||\pi\{x(t,\xi)\}||_{L^2(\Omega)} \rightarrow 0$ exponentially as $t \rightarrow \infty$. \hfill $\square$
\end{proposition}

\begin{proof}
First define $\tilde{x}=\pi\{x\}$. Note that 
\begin{equation}
\frac{\partial \tilde{x}}{\partial t} = \pi\{f(x)\} + \mathcal{L} x.
\end{equation}
Consider the candidate Lyapunov functional $V(\tilde{x}) = \frac{1}{2} \langle \tilde{x}, P \tilde{x} \rangle_{L^2(\Omega)}$. Differentiating, we have:
\begin{equation}
\dot{V}(\tilde{x}) \leq \langle \tilde{x}, P \pi \{\tilde{x}\} \rangle_{L^2(\Omega)} + \langle \tilde{x}, P \mathcal{L}x \rangle_{L^2(\Omega)}.
\label{eq:lyapunovderivative}
\end{equation}
We consider the expansion:
\begin{equation}
\langle \tilde{x}, P \mathcal{L}x \rangle_{L^2(\Omega)} = \sum_{k=1}^n \langle \tilde{x}, P E_k \mathcal{L} x \rangle_{L^2(\Omega)},
\end{equation}
and note that
\begin{equation}
\langle \tilde{x}, P E_k \mathcal{L} x \rangle_{L^2(\Omega)} = \langle \tilde{x}, P E_k \mathcal{L}_k x \rangle_{L^2(\Omega)},
\end{equation}
where the linear operator $\mathcal{L}_k$ is defined:
\begin{equation}
\mathcal{L}_k x = [L_k x_k \,\,\, \ldots \,\,\, L_k x_k \,\,\, \ldots \,\,\, L_k x_k]^T.
\end{equation}
From the condition (\ref{eq:coroconditions2}), we know there exists a matrix $Q_k$ such that $Q_k^T Q_k = \frac{1}{2}(PE_k + E_kP)$.  Substituting, we have:
\begin{equation}
\langle \tilde{x}, P E_k \mathcal{L}_k x \rangle_{L^2(\Omega)} = \langle Q_k \tilde{x}, Q_k \mathcal{L}_k \tilde{x} \rangle_{L^2(\Omega)} = \langle y_k, \mathcal{L}_k y_k \rangle_{L^2(\Omega)},
\end{equation}
where $y_k = Q_k \tilde{x}$. Consider the following identity:
\begin{equation}
\nabla \cdot (y_{k,i} A_k \nabla y_{k,i}) = \nabla y_{k,i} \cdot (A_k \nabla y_{k,i}) + y_{k,i} \nabla \cdot (A_k \nabla y_{k,i}).
\end{equation}
Integrating both sides, noting the Neumann boundary conditions, and applying the divergence theorem, we see that the left hand side of the integrated identity is zero. We then have:
\begin{equation}
\int_{\Omega} y_{k,i} \nabla \cdot (A_k \nabla y_{k,i}) \, d \xi = - \int_{\Omega} \nabla y_{k,i} \cdot (A_k \nabla y_{k,i}) \, d \xi.
\end{equation}
Noting that $\int_{\Omega} y_k \, d\xi = Q_k \int_{\Omega} \tilde{x} \, d \xi = 0$, we apply Lemma \ref{lemma1}:
\begin{equation}
\int_{\Omega} \nabla y_{k,i} \cdot (A_k \nabla y_{k,i}) \, d \xi \geq \lambda_2^{(k)} \int_{\Omega} y_{k,i}^2 \, d \xi,
\end{equation}
where $\lambda_2^{(k)}$ is the second Neumann eigenvalue of $L_k$. Substituting, we have:
\begin{equation}
\begin{aligned}
\langle \tilde{x}, P \mathcal{L}x \rangle_{L^2(\Omega)} &= \sum_{k=1}^n \langle y_k, \mathcal{L}_k y_k \rangle_{L^2(\Omega)} \\
&\leq - \sum_{k=1}^n \lambda_2^{(k)} \langle y_k,y_k \rangle_{L^2(\Omega)} \\
&= - \sum_{k=1}^n \lambda_2^{(k)} \langle \tilde{x}, P E_k \tilde{x} \rangle_{L^2(\Omega)}.
\end{aligned}
\end{equation}
After adding and subtracting $f(\bar{x})$ to the first term on the right hand side of (\ref{eq:lyapunovderivative}), we arrive at:
\begin{equation}
\begin{aligned}
\dot{V} &\leq \langle \tilde{x}, P f(x)-f(\bar{x}) \rangle_{L^2(\Omega)} - \sum_{k=1}^n \lambda_2^{(k)} \langle \tilde{x}, P E_k \tilde{x} \rangle_{L^2(\Omega)} \\
&= \left \langle \tilde{x}, P \left (f(x)-f(\bar{x}) - \sum_{k=1}^n \lambda_2^{(k)} E_k \tilde{x} \right ) \right \rangle_{L^2(\Omega)}.
\end{aligned}
\end{equation}
An application of the mean value theorem to $f(x)-f(\bar{x})$ taken together with condition (\ref{eq:coroconditions1}) gives:
\begin{equation}
\begin{aligned}
\dot{V} &\leq \int_{0}^{1} \int_{\Omega} \tilde{x}^T P \left (J(\bar{x}+s\tilde{x}) - \sum_{k=1}^n \lambda_2^{(k)} E_k \right) \tilde{x} \, d \xi \, d s \\
&\leq \int_{0}^{1} \int_{\Omega} -\frac{\epsilon}{2} \tilde{x}^T \tilde{x} \, d \xi \, d s \leq - \frac{\epsilon}{\lambda_{\text{max}}(P)} V,
\end{aligned}
\end{equation}
which concludes the proof.
\end{proof}
\hfill $\square$

\section{Conclusion}
We have derived Lyapunov inequality conditions that guarantee spatial uniformity in the solutions of compartmental ODEs and reaction-diffusion PDEs even when the diffusion terms vary between species. We have used convex optimization to develop tests using linear matrix inequalities that imply the inequality conditions, and have applied the tests to coupled ring oscillator circuits. In future work, we will study different spatial domains with boundary conditions different from the Neumann condition as well as apply the conditions we have derived to biological reaction-diffusion networks.

\bibliographystyle{IEEEtran}
\bibliography{mybibGenSync}

\end{document}